\documentclass[11pt]{amsart}
\usepackage{amsmath,amssymb,amsthm}
\usepackage[latin1]{inputenc}
\usepackage{version,tabularx,multicol}
\usepackage{graphicx,float,psfrag}

\newcommand{\reff}[1]{(\ref{#1})}

\theoremstyle{plain}
\newtheorem{theo}{Theorem}[section]
\newtheorem{cor}[theo]{Corollary}
\newtheorem{prop}[theo]{Proposition}
\newtheorem{lem}[theo]{Lemma}

\theoremstyle{remark}
\newtheorem{rem}[theo]{Remark}

\newcommand{\cb}{{\mathcal B}}

\newcommand{\cl}{{\mathcal L}}
\newcommand{\cn}{{\mathcal N}}

\newcommand{\cp}{{\mathcal P}}

\newcommand{\calr}{{\mathcal R}}

\newcommand{\ct}{{\mathcal T}}
\newcommand{\cu}{{\mathcal U}}

\newcommand{\E}{{\mathbb E}}

\newcommand{\N}{{\mathbb N}}
\renewcommand{\P}{{\mathbb P}}

\newcommand{\T}{{\mathbb T}}

\newcommand{\bm}{\mathbf m}

\newcommand{\bE}{{\mathbf E}}
\newcommand{\bP}{{\mathbf P}}
\newcommand{\bs}{{\mathbf s}}
\newcommand{\bt}{{\mathbf t}}

\newcommand{\ind}{{\bf 1}}

\newcommand{\Card}{{\rm Card}\;}

\newcommand{\val}[1]{\mathop{\left| #1 \right|}\nolimits}
\newcommand{\inv}[1]{\mathop{\frac{1}{ #1}}\nolimits}
\newcommand{\expp}[1]{\mathop {\mathrm{e}^{ #1}}}

\newcommand{\lb}{[\![}
\newcommand{\rb}{]\!]}

\title{$\beta$-coalescents and stable Galton-Watson trees}

\date{\today}
\author{Romain Abraham} 

\address{
Romain Abraham,
Laboratoire MAPMO, CNRS, UMR 7349,
F\'ed\'eration Denis Poisson, FR 2964,
 Université d'Orléans,
B.P. 6759,
45067 Orléans cedex 2,
France.
}
  
\email{romain.abraham@univ-orleans.fr}

\author{Jean-François Delmas}

\address{
Jean-Fran\c cois Delmas,
Université Paris-Est, \'Ecole des Ponts, CERMICS, 6-8
av. Blaise Pascal, 
  Champs-sur-Marne, 77455 Marne La Vallée, France.}

\email{delmas@cermics.enpc.fr}

\begin{document}

\begin{abstract}
Representation of coalescent process using pruning of trees has been
used by Goldschmidt and Martin for the Bolthausen-Sznitman coalescent and by
Abraham and Delmas for the $\beta(3/2,1/2)$-coalescent. 
By considering a pruning procedure on stable Galton-Watson tree with
$n$ labeled leaves, we give a representation of the discrete 
$\beta(1+\alpha,1-\alpha)$-coalescent, with $\alpha\in [1/2,1)$ starting
from the trivial partition of the $n$ first integers. The construction can
also be made directly on the stable continuum L\'evy tree, with
parameter $1/\alpha$, simultaneously for all $n$. This representation
allows to use results on the asymptotic number of coalescence events to
get the asymptotic number of cuts in stable Galton-Watson tree (with
infinite variance for the offspring distribution) needed to isolate the
root. Using convergence of the stable Galton-Watson tree conditioned to
have infinitely many leaves, one can get the asymptotic distribution of
blocks in the last coalescence event in the
$\beta(1+\alpha,1-\alpha)$-coalescent. 
\end{abstract}

\maketitle

\section{Introduction}
\subsection{Framework}
\label{sec:framework}
The idea of constructing  coalescent processes by pruning discrete trees
arises   first  in   \cite{gm:rrtbsc}   where  the   Bolthausen-Sznitman
coalescent  is constructed by  a uniform  pruning of  the branches  of a
random  recursive tree, see  also \cite{s:debsc} and  \cite{fs-j:mcsbsc} for
applications of such a representation. The
same  kind of  ideas  has been  used  in \cite{ad:cbcpbt}  to construct  a
$\beta(3/1,1/2)$-coalescent  process  using  a  uniform pruning  of  the
branches  of a  uniform random  binary tree.  This construction  is also
closely  related to  Aldous's continuum  random tree.  The goal  of this
paper  is  to  extend  this  result  by  applying  a  pruning  at  nodes
(introduced  in   \cite{ad:falpus}  in  a  continuous   setting  and  in
\cite{adh:pgwttvmp}  in  a discrete  setting)  to  a  stable L\'evy  tree,
obtaining   a    $\beta(1+\alpha,1-\alpha)$-coalescent   process,   with
$1/2\leq \alpha<1$.

Let $\Lambda$ be a   finite  measure  on
$[0,1]$. 
A $\Lambda$-coalescent $(\Pi(t),t\ge 0)$ is a Markov process which takes
values in the set of partitions of $\N^*=\{1, 2, \ldots\}$ introduced in
\cite{p:cmc} for  coalescent processes with possible
multiple collisions. It is defined via  the transition rates of its restriction
$\Pi^{[n]}=(\Pi^{[n]}(t),t\ge 0)$ to the  $n$ first integers: if $\Pi^{[n]}(t)$ is
composed of $b$  blocks, then $k$ $(2\le k\le  b)$ fixed blocks coalesce
at rate:
\begin{equation}
   \label{eq:lambda-bn}
\lambda_{b,k}=\int_0^1u^{k-2}(1-u)^{b-k}\Lambda (du).
\end{equation}
In particular a coalescence event happens at rate:
\begin{equation}
   \label{eq:lambda-n}
\lambda_b=\sum_{k=2}^b \binom{b}{k} \lambda_{b,k}. 
\end{equation}
We take the convention $\lambda_1=0$. We         also        define        the         discrete        process
$\Pi_\text{dis}^{[n]}=(\Pi_\text{dis}^{[n]}   (k),  k\in  \N)$   as  the
different successive states of  the process $\Pi^{[n]}$ until it reaches
the absorbing  state (which is  the trivial partition consisting  in one
block) and afterward the discrete process remains constant.

As examples  of $\Lambda$-coalescents, let us  mention: 
\begin{itemize}
   \item the Kingman's coalescent with $\Lambda(dx)=\delta_0(dx)$,  see
     \cite{k:c},
   \item the  Bolthausen-Sznitman
coalescent  with $\Lambda(dx)=\ind_{(0,1)}(x)dx$, see  \cite{bs:rpcacm},
   \item the $\beta$-coalescents   where $\Lambda(dx)$  is  (up to a
     multiplicative constant) the  $\beta(a,b)$ distribution. In the
     case of the $\beta(1+\alpha, 1-\alpha)$-coalescent, that is 
$\Lambda(dx)=(x/(1-x))^\alpha \, dx$, see \cite{bbcemsw:asbbc,
  bbs:bcsrt} for $-1<\alpha<0$. The case $\alpha=0$ corresponds to the
Bolthausen-Sznitman coalescent, while the limit case $\alpha=-1$ formally
corresponds to the Kingman's coalescent. 
For the $\beta(1+\alpha,-\alpha)$-coalescent, with $-1<\alpha<0$  see 
\cite{fh:cbi-gfvi}.  
\end{itemize} 
We  refer to the survey  \cite{b:rpct} for further
results on coalescent processes.

Let $\alpha\in [1/2,1)$. We consider  a  critical  Galton-Watson (GW) tree  $T$ with
offspring  distribution  characterized by  its  generating function  for
$r\in [0,1]$:
\begin{equation}\label{eq:def-g}
g(r)=r+\alpha(1-r)^{1/\alpha}.
\end{equation}
This GW tree arises as the shape of the sub-tree of a stable L\'evy tree
with index $\gamma=1/\alpha$ generated by  leaves chosen in a Poissonian
manner,  see  \cite{dlg:rtlpsbp}, Theorem  3.2.1.  We  shall call  these
random trees the  stable GW trees with parameter $\gamma$.  We denote by
$\bP$ the  distribution of $T$.  If  $x$ is a  node of $T$ we  denote by
$k_x(T)$   the   number   of   offsprings   of   $x$.    If   $k_x(T)=0$
(resp. $k_x(T)>0$), then $x$ is called  a leaf (resp.  an internal node)
of $T$. We denote by $L(T)$ the number of leaves of the tree $T$. Since $g'(0)=0$, we get  that a.s. $k_x(T)\neq 1$ for all $x\in
T$.  We denote by $\bP_n$ the law of $T$ conditioned to have exactly $n$
leaves.   Under $\bP_n$,  we label  the leaves  of $T$  from $1$  to $n$
uniformly  at random,  independently of  $T$, and  then we  consider the
following pruning procedure which is derived from \cite{adh:pcrtst}, see
Section \ref{sec:d-tree-valued}.  Choose an  internal node  $x_1$ (which
has at least 2 children) at random with probability:
\[
\frac{k_{x_1}(T)-1}{L(T)-1}\cdot
\]
This internal node separates the tree into two subtrees: the fringe
sub-tree $T_{x_1}$ rooted at $x_1$ that consists of all nodes of $T$
that have $x_1$ on their lineage to the root (including $x_1$), and
the set $T\setminus T_{x_1}$ which is still a tree.
We set $T_{(1)}=(T\setminus T_{x_1})\cup\{x_1\}$ which is the new tree
we work with. All the leaves of $T_{(1)}$ except $x_1$ are leaves of
$T$ and they keep their label. Notice that $x_1$ is a new leaf of $T_{(1)}$ and we label it by the block
(i.e. the sequence) of labels of the leaves of $T_{x_1}$. We then
iterate the procedure on the tree $T_{(1)}$ and so on until the root is
chosen (see Figure \reff{fig:pruning}).

This    pruning   procedure    defines   a    discrete    time   process
$\Pi_\text{GW}^{[n]}=(\Pi_\text{GW}^{[n]}(k),k\in \N)$  taking values in
the    set    of    partitions    of    the    $n$    first    integers,
$\Pi_\text{GW}^{[n]}(k)$ being  the set of  labels of the leaves  of the
tree $T_{(k)}$ obtained after the $k$-th  cut.

\begin{center}
\begin{figure}[ht]
\includegraphics[width=5cm]{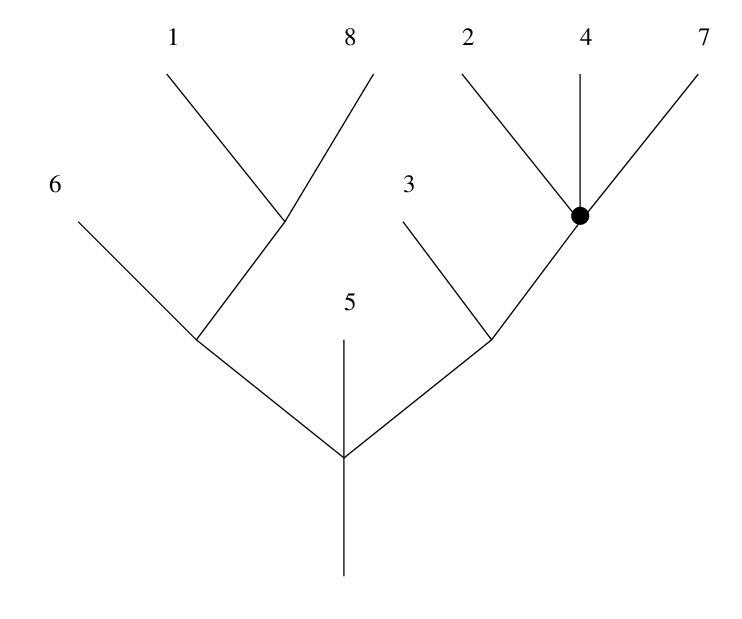}
\includegraphics[width=5cm]{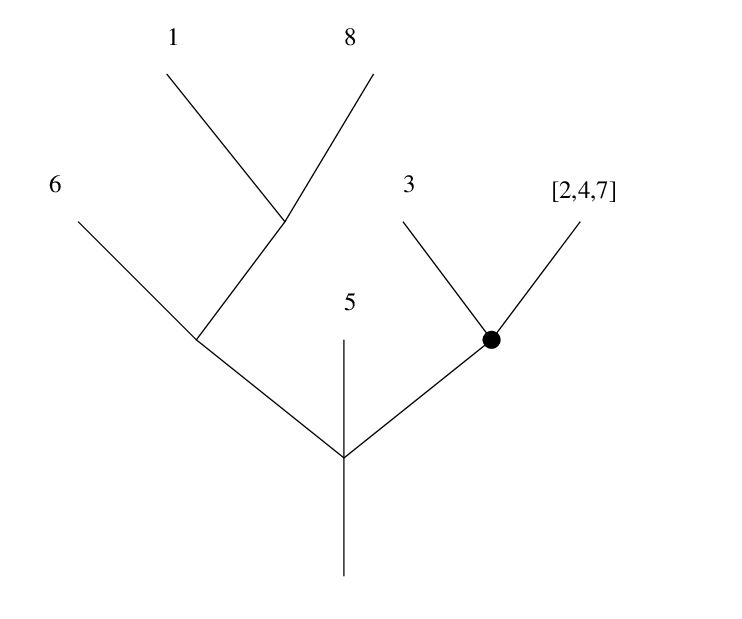}\\
\includegraphics[width=5cm]{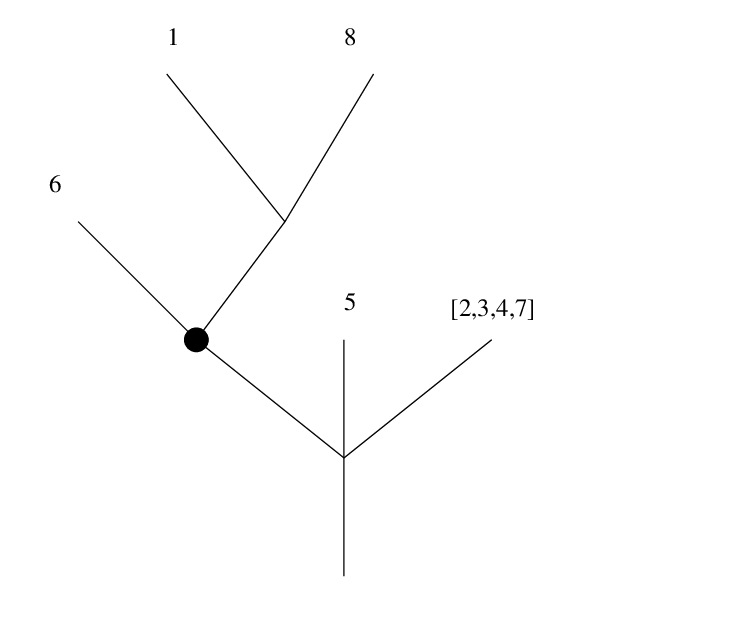}
\includegraphics[width=5cm]{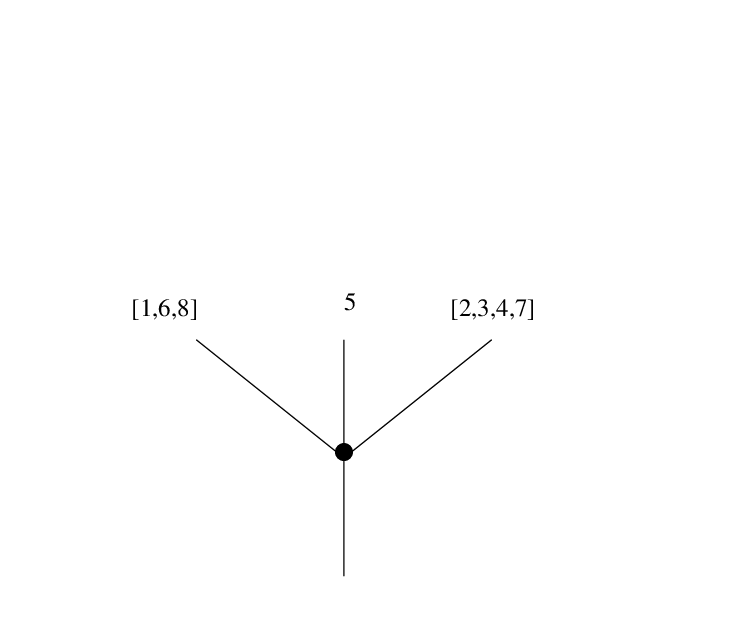}
\includegraphics[width=5cm]{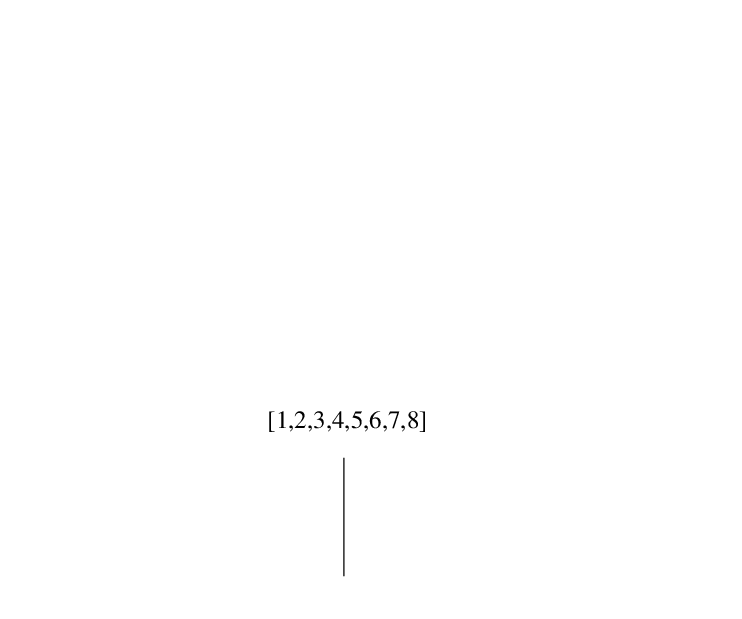}
\caption{The pruning at node of a given tree. The bold internal node
  corresponds to the next chosen node.}\label{fig:pruning}
\end{figure}
\end{center}

\subsection{Main result}
The process $\Pi_\text{GW}^{[n]}$ is then a coalescent process starting
from the trivial partition consisting of singletons and blocks merge
together as time goes by. Its law is given in the next theorem.

\begin{theo}\label{th:main}
We set $\alpha=\inv{\gamma}\in[1/2,1)$.
The process $\Pi_\text{GW}^{[n]}$ is distributed under $\bP_n$ as
$\Pi^{[n]}_\text{dis}$ for the
 $\beta(1+\alpha,1-\alpha)$-coalescent with coalescent measure:
\begin{equation}
   \label{eq:def-Lambda}
\Lambda(dx)=\left(\frac{x}{1-x}\right)^\alpha \, dx.
\end{equation}
\end{theo}

\begin{rem}
Notice that the process $\Pi^{[n]}_\text{dis}$ is discrete in time and
thus characterizes the coalescent measure up to a multiplicative
constant. It is possible to construct the continuous-time coalescent process
$\Pi^{[n]}$ associated with the measure $\Lambda$ given by Equation
\reff{eq:def-Lambda} from the process  $\Pi_\text{GW}^{[n]}$ by adding
exponential times between the successive states of this process.
More precisely, recall the definitions of the transitions rates
$\lambda_{b,k}$ of Equation \reff{eq:lambda-bn} and of the jump rates
$\lambda_b$ of Equation \reff{eq:lambda-n}. Let $(\tau_k)_{k\in\N}$ be
a sequence of independent random variables such that, conditionally
given the process $\Pi_\text{GW}^{[n]}$, the random variable $\tau_k$
is exponentially distributed with parameter $\lambda_{\ell_k}$ where
$\ell_k$ is the number of blocks of the partition
$\Pi_\text{GW}^{[n]}(k)$, with the convention that $\tau_k=+\infty $
  if $\ell_k=1$. Then we set
\[
\tilde\Pi^{[n]}(t)=\Pi_\text{GW}^{[n]}(k)\quad\mbox{ if
}\sum_{i=0}^{k-1}\tau_i\le t<\sum_{i=0}^k\tau_i.
\]
As a direct consequence of Theorem \ref{th:main} and the definition of
a $\Lambda$-coalescent, we get that the processes $\Pi^{[n]}$ and
$\tilde\Pi^{[n]}$ have the same distribution.
\end{rem}

One major  drawback of this construction  is that we  define the process
for  fixed $n$  and  not simultaneously  for  all $n$.   However, as  in
\cite{ad:cbcpbt},    we    can    construct   directly    the    process
$(\Pi(\theta),\theta\ge 0)$  taking values in  the set of  partitions of
the integers using  the pruning of a L\'evy  continuum random tree. More
precisely,  we  consider  the   weighted  stable  L\'evy  tree  $(\ct,d,
\bm^\ct)$      associated     with      the      branching     mechanism
$\psi(\lambda)=\lambda^\gamma$ for $\gamma\in(1,2)$ (the case $\gamma=2$
is  studied in \cite{ad:cbcpbt}  and requires  a different  pruning). We
recall that  $\ct$ is  a real  tree and that  $\bm^\ct$ corresponds  to a
uniform  measure  on  the   leaves  of  $\ct$,  see  \cite{dlg:rtlpsbp},
\cite{dlg:pfalt} and  also \cite{adh:etiltvp} more  specifically for the
space of  weighted real trees.   We work under the  so-called normalized
excursion  measure $\N^{(1)}$  under  which $\bm^\ct$  is a  probability
measure.    We   consider   given   $\ct$   the   pruning   defined   in
\cite{ad:falpus}: to each branching point  $x$ of $\ct$ we can associate
a  ``mass'' $\Delta_x$ of  this node,  which intuitively  represents the
size of its progeny, and  a random variable $E_x$ which is exponentially
distributed with  parameter $\Delta_x$. This  random variable represents
the time  at which  the node $x$  is cut.  When we cut  such a  node, we
remove  the sub-tree above  it.  Let  $\ct_\theta$ denote  the continuum
random sub-tree obtained  at time $\theta\geq 0$. We  define a coalescent
process  using the  usual paintbox  procedure.  Let  $(U_i,i\in\N^*)$ be
independent   random  variables   with   distribution  $\bm^\ct$   under
$\N^{(1)}$.   We  define  a   partition  of  $\N^*$  at  time  $\theta$,
$\Pi_\text{Lévy}(\theta)$ by saying that two integers $i$ and $j$ belong
to the same block of $\Pi_\text{Lévy}(\theta)$ if and only if the random
variables  $U_i$ and  $U_j$  have a  leaf  of $\ct_\theta$  as a  common
ancestor. Intuitively this means that $U_i$ and $U_j$ belong to the same
sub-tree attached  above $\ct_\theta$.  This defines  a coalescent process
$\Pi_\text{Lévy}=(\Pi_\text{Lévy}(\theta),  \theta\geq 0)$.  We  are now
interested in its  discrete  (in  time) restriction  to  the $n$  first
integers.   Let  $\Pi_\text{Lévy}^{[n]}=(\Pi_\text{Lévy}^{[n]}(k),  k\in
\N)$ be the discrete  process associated with $\Pi_\text{Lévy}$ restricted
to the $n$ first integers until it reaches the absorbing state (which is
the  trivial  partition consisting  in  one  block)  and which  afterward
remains constant.

By construction,  and thanks to Theorem 3.2.1  in \cite{dlg:rtlpsbp}, we
can  deduce  that  under  $\N^{(1)}$, the  discrete  coalescent  process
$\Pi_\text{Lévy}^{[n]} $  is distributed as  $\Pi_\text{GW}^{[n]}$ under
$\bP_n$. In fact, we have the following stronger result.

\begin{theo}\label{th:main2}
We set $\alpha=\inv{\gamma}\in (1/2,1)$.
Under $\N^{(1)}$, the processes $(\Pi_\text{Lévy}^{[n]},n\in \N^*)$
associated with the Lévy tree with branching mechanism
$\psi(\lambda)=\lambda^\gamma$ are
distributed as $(\Pi_\text{dis}^{[n]}, n\in \N^*)$ associated with the Lévy
measure $\Lambda(dx)=(x/1-x)^\alpha \, dx$. 
\end{theo}

\begin{rem}
Although the process $\Pi_\text{L\'evy}$ is a continuous-time process
like $\Pi_\text{GW}$, it is not a coalescent process under $\N^{(1)}$ as
for instance the time of the
first coalescence event in $\Pi_\text{Lévy}^{[n]}$ is not
exponentially distributed, see Corollary \ref{cor:first_event}.

We conjecture that there exists a random time-change $(R(t),t\ge 0)$
such that the process $(\Pi_\text{Lévy}(R(t)),t\ge 0)$ is indeed  under
$\N^{(1)}$ a $\beta(1+\alpha,1-\alpha)$-coalescent, but we
have no guess on what this time change could be.
\end{rem}

\begin{rem}
   \label{rem:a=1-alpha}
   Let  us remark  that the  $\beta(1+\alpha,  1-\alpha)$-coalescent we
   obtain is also a  $\beta(2-a,a)$-coalescent (with $a=1-\alpha$) as in
   \cite{bbs:bcsrt} but  with a different range for  $a$. The difference
   between the two cases  is that in \cite{bbs:bcsrt} $\alpha\in (-1,0)$
   and the coalescent  process comes down from infinity  (i.e. for every
   positive time  $\theta$, the partition $\Pi(\theta)$  contains only a
   finite number of blocks) whereas  in our case $\alpha\in (1/2,1)$ the
   process always contains an infinite number of singletons (also called
   ``dust'').
\end{rem}

\begin{rem}
Let us remark that the pruning procedure described above is the same
as in \cite{m:ssfdstIIsn} used  to construct the Miermont's
self-similar fragmentation process (see also
\cite{ad:falpus}). However, the time reversal of the process
$\Pi_\text{Lévy}$ is not Miermont's fragmentation as once a sub-tree
is cut and discarded, it is no more considered in our construction
whereas it undergoes some others fragmentations in Miermont's
construction. There are still some strong connections. For instance,
the tree $\ct_\theta$ is linked with a tagged fragment in the
fragmentation, see \cite{ad:falpus} Theorem 1.5 and Proposition 1.7 for the distribution
of the tree $\ct_\theta$ and  for the
distribution of a tagged fragment in Miermont's fragmentation.
\end{rem}

\subsection{Number of cuts needed to isolate the root in a stable GW
  tree} 

Using the above link between Galton-Watson trees and
$\beta$-coalescents, known results in one field translate immediately
in the other field giving sometimes new results. In that direction,
we first focus on how known asymptotics on the number of coalescence
events yield new results on the number  of  cuts needed  to
isolate the  root in a stable GW  tree with $n$ leaves.  

The original problem of cutting randomly a rooted tree arises first in Meir and
Moon \cite{mm:cdrt}. Given a rooted tree $T_n$ with $n$ edges, select
an edge uniformly at random (notice that this is not exactly our
pruning procedure) and delete the subtree not containing the
root attached to this edge. On the remaining tree, iterate this
procedure until only the edge attached to the root is left. We denote
by $\tilde Z_n$ the number of edge-removals needed to isolate the root. The
problem is then to study asymptotics of this random number $\tilde Z_n$,
depending on the law of the initial tree $T_n$.

In the original paper \cite{mm:cdrt}, Meir and Moon considered Cayley
trees and obtained asymptotics for the first two moments of
$X_n$. Limits in distribution were then obtained, see for instance Panholzer
\cite{p:cdvst} for some simply generated trees, Drmota,
Iksanov, Möhle and Roesler \cite{dimr:ldncnirrrt} for random
recursive trees, Holmgren \cite{h:rrcbsc} for binary search trees,
Bertoin \cite{b:ft} for Cayley trees. In \cite{j:rcrdrt}, Janson
focuses on conditioned Galton-Watson trees associated with critical
offspring  distributions with finite  variance and proves that 
\[
\tilde Z_n/\sqrt{n}
\xrightarrow[n\rightarrow +\infty ]{(d)} \;
\tilde Z,
\]
where the random  variable  $\tilde Z$  has  Rayleigh  distribution  with  density
$x\expp{-x^2/2}\ind_{\{x>0\}}$, and can be explicitly constructed using a pruning procedure
on the Brownian continuum random tree (which corresponds to the cases
$\gamma=2$ in our setting), see \cite{ad:rpcrt}. In particular $\tilde
Z$ is distributed as  the height of  a random leaf  of the
Brownian continuum random tree. See also \cite{abbh:ctmc,bm:ctlgwtbcrt} for further work on
cutting randomly rooted trees.

Notice that the
reproduction  law  for stable  GW  tree  has  an infinite  variance  for
$\alpha\in (1/2,1)$, and the uniform pruning does not seem to be
adapted to isolate the root. For this reason, we consider the pruning
procedure developed in Section \ref{sec:framework} to tackle the
infinite variance case. So, let $Z_n$ be the number of cuts, using
this procedure, needed to isolate the root of a stable GW tree:
\[
Z_n=\inf\{k;\;  \Pi^{[n]}_\text{GW}(k)= \{\{1, \ldots, n\}\}\}.
\]
Notice that  for $r$-ary  trees, since all  the internal nodes  have the
same degree the cutting  procedure given in Section \ref{sec:framework},
corresponds to  choose an internal  node uniformly.

We immediately deduce from  asymptotics  of the number
of  coalescence events  in  $\beta$-coalescent (see Corollary 1
\cite{hm:sslnmc}, see also \cite{gim:abc}, Table 1 for a summary of
all the results concerning $\beta$-coalescents),   the
following  result  which extends part of  the  result  in
\cite{j:rcrdrt}  to GW tree with infinite variance of the reproduction
law. 

\begin{cor}
   \label{cor:momentZ} 
Let $\alpha=1/\gamma\in [1/2,1)$. 
We have the following convergence in distribution:
\[
n^{\alpha-1} Z_n\xrightarrow[n\rightarrow+\infty ] {(d) }Z,
\]
with the distribution of $Z$ characterized by, for $n\in \N^*$:
\[
\E\left[Z^n\right]=\alpha^n \frac{\Gamma(n+1)
  \Gamma(1-\alpha)}{\Gamma((n+1)(1-\alpha))}\cdot 
\]
\end{cor}

Let us insist on the fact that this corollary does not need any proof
as this is just a translation of known results on $\beta$-coalescents
using our links with GW trees, only the moment computation needs some
explanations and is done in  Section \ref{sec:proof-cut},

The distribution  of $Z$ corresponds to the  expected limit distribution
in the Conjecture that is stated at the end of the introduction in \cite{ad:farplt} for the number  of cuts needed to
isolate the  root in general GW  trees. (Notice that  in the conjecture,
one choose  an internal  node $x\in T$  with probability  proportional to
$k_x(T)$ whereas  in Section \ref{sec:framework} one  choose an internal
node  $x\in   T$  with  probability  proportional   to  $k_x(T)-1$.)   In
particular, $Z$  is distributed as  the height of  a random leaf  of the
normalized       Lévy      tree      with       branching      mechanism
$\psi(\lambda)=\lambda^\gamma$.

\subsection{Number of blocks in the last coalescence event}

On the other hand, we can use results on  GW
trees conditioned to have an infinite number of leaves (which is very
close to Kesten's result on GW tree conditionally on the non
extinction, see \cite{ck:rncpccgwta} Theorem 3.1 or \cite{ad:llcgwtisc} Proposition 4.6)
to get asymptotics on the number $B_n$ of blocks involved is the last coalescence
event of $\Pi^{[n]}$.

The proof of the following Proposition is given in Section
\ref{sec:proof-bn}. 
\begin{prop}
   \label{prop:cv-bn}
Let $\alpha=1/\gamma\in [1/2,1)$. 
We have the following convergence in distribution:
\[
B_n\xrightarrow[n\rightarrow+\infty ] {(d) }B,
\]
with the distribution of $B$ given by its generating function
$\varphi_\alpha(r)=\E\left[r^B\right]$, with for $r\in
[0,1]$: 
\begin{equation}
   \label{eq:def-f-a}
\varphi_\alpha(r)
= (1-\alpha)r \int_0^1 \frac{dx}{1-(1-x)^\alpha}
\left(\inv{(1-rx)^\alpha} -1\right).
\end{equation}
\end{prop}

See also \cite{ad:cbcpbt} for more results in this
direction when $\alpha=1/2$ including the number of singletons involved
in the last coalescence event as well as a closed form for $\varphi_{1/2}$.

\begin{rem}
After we first posted this paper on arXiv, H\'enard proved in
\cite{h:fl} Theorem 3.5 that Equation \reff{eq:def-f-a} remains valid for all
$\beta(1+\alpha,1-\alpha)$-coalescents with $\alpha\in (-1,1)$ (taking
the limit when $\alpha=0$).

For $\alpha=0$, the $\beta(1+\alpha,1-\alpha)$-coalescent corresponds to
the  Bolthausen-Sznitman   coalescent,  and  thus  $\varphi_0$   is  the
generating  function of  the asymptotic  number  of blocks  of the  last
coalescence   event   in   the  Bolthausen-Sznitman   coalescent   whose
distribution  is   given  in   Theorem  3.1   and  Proposition   3.2  of
\cite{gm:rrtbsc}.

As $\alpha$ goes down to $-1$, we recover the Kingman's coalescent as a
limit. We also get $\varphi_{-1}(r)=r^2$ and notice that
$\varphi_{-1}$ is trivially the
generating function of the number of blocks of the last (in fact all)
coalescence event in the Kingman's coalescent, as all the coalescence events
are binary. 
\end{rem}

\subsection{Organization of the paper}
Section \ref{sec:pruning} gives a representation of the pruning at node 
procedure for GW tree in continuous time motivated by
\cite{adh:pcrtst}. This procedure corresponds in fact  to the one presented in
Introduction, Section \ref{sec:framework}. Section \ref{sec:proof-1} is
devoted to the proof of  Theorem \ref{th:main}. Section
\ref{sec:proof-2} devoted to the proof of Theorem
  \ref{th:main2} is more technical as it relies on continuum random
Lévy trees and the pruning of such trees as developed in
\cite{ad:falpus}. Eventually Sections \ref{sec:proof-cut}
and \ref{sec:proof-bn} are devoted to the proofs of Propositions
\ref{cor:momentZ} and \ref{prop:cv-bn}. 

\section{Pruning at node of discrete GW  trees}
\label{sec:pruning}

\subsection{Discrete trees}

Let us recall here the formalism for ordered discrete trees. We set
\[
\cu=\bigcup _{n\ge 0}{(\N^*)^n}
\]
the set  of finite  sequences of positive  integers with  the convention
$(\N^*)^0=\{\emptyset\}$.  For  $u\in  \cu$  let  $|u|$  be  the  length  or
generation  of  $u$   defined  as  the  integer  $n$   such  that  $u\in
(\N^*)^n$. If $u$ and $v$ are  two sequences of $\cu$, we denote by $uv$
the concatenation of the two  sequences, with the convention that $uv=u$
if $v=\emptyset$ and  $uv=v$ if $u=\emptyset$.  The set  of ancestors of
$u$ is the set:
\begin{equation}
   \label{eq:Au}
A_u=\{v\in \cu; \text{there exists $w\in \cu$ such that $u=vw$}\}.
\end{equation}

A discrete tree $\bt$ is a subset of $\cu$ that satisfies:
\begin{itemize}
\item $\emptyset\in\bt$,
\item If  $u\in\bt$, then $A_u\subset \bt$. 
\item For every $u\in \bt$, there exists a non-negative integer
  $k_u(\bt)$ such that, for all positive integer $i$, $ui\in \bt$ iff $1\leq i\leq k_u(\bt)$. 
\end{itemize}

The integer $k_u(\bt)$ 
represents the number of offsprings of the node $u$ in the tree
$\bt$.  We define $\cl(\bt)$ the set of leaves of $\bt$ and $\cn(\bt)$ 
the set of internal nodes of $\bt$ by:
\[
\cl(\bt)=\{u\in \bt,\ k_u(\bt)=0\}
\quad\text{and}\quad 
\cn(\bt)=\bt\setminus \cl(\bt).
\]
Let $L(\bt)=\Card (\cl(\bt))$ be the number of leaves of the tree $\bt$,
and notice that:
\begin{equation}\label{eq:leaves-nodes}
L(\bt)-1= \sum_{u\in \cn(\bt)} (k_u(\bt)-1 ).
\end{equation}

We denote by  $\T$ the set of discrete trees and  by $\T_n=\{ \bt\in \T;
L(\bt)=n\}$ the set of discrete trees with $n$ leaves.

\subsection{A discrete tree-valued process}
\label{sec:d-tree-valued}

We consider the pruning procedure developed in
\cite{adh:pgwttvmp}. Let $\bt\in\T$.
Under some probability measure $\bP^\bt$, we consider
a family of marks $(\xi_u,u\in \cu)$ of independent non-negative real random
variables (possibly infinite) such that:
\begin{itemize}
   \item $\bP^\bt$-a.s. $\xi_u=+\infty
$ if $u\not \in
\bt$ or if $u\in \bt $ and $k_u(\bt)\in \{0,1\}$, 
\item $\bP^\bt(\xi_u\ge \theta)=(1+\theta)^{1-k_u(\bt)}$ 
if $u\in \bt$ and  $k_u(\bt)\geq 2$.
\end{itemize} 

At time $\theta$, we define the pruned tree $\cp_\theta(\bt)$ as the
sub-tree given by:
\[
\cp_\theta(\bt)=\{u\in \bt; \, \xi_v>\theta \text{ for all } v\in A_u,
\, v\neq u\}.
\]
In particular, we always have $\emptyset\in \cp_\theta(\bt)$.

For $u\in \cn(\bt)$,
let $D_{u}$ be the event that $u$ is marked first, that is:
$$D_u=\{\xi_u=\min_{v\in\cn(\bt)}\xi_v\}.$$

\begin{lem}
   \label{lem:Ax}
We suppose that $L(\bt)\ne 1$. Let $u\in \cn(\bt)$. We have:
\[
\bP^\bt(D_{u})= \frac{k_u(\bt)-1}{L(\bt)-1}\cdot
\]
\end{lem}

This lemma implies that the cutting procedure given in Section
\ref{sec:framework}, corresponds to the successive states of the process
$(\cp_\theta(\bt), \theta\geq 0)$.

\begin{proof}
We have, using \reff{eq:leaves-nodes} for the last equality:
\begin{align*}
\bP^\bt(D_{u})
&=\bP^\bt(\xi_{u} \leq  \xi_{v} \quad \forall v\neq u, v\in \cn(\ct))\\
&= \bE^\bt\left[(1+\xi_{u})^{-\sum_{v\neq u, v \in \cn(\bt)} (k_v(\bt)-1
    )}\right]\\ 
&= (k_{u}(\bt) -1) \int_{[0, +\infty )} (1+\theta)^{-\sum_{v\in \cn(\bt)}
  (k_v(\bt)-1 )-1}\; d\theta\\ 
&= \frac{k_u(\bt)-1}{\sum_{v\in\cn(\bt)} (k_v(\bt)-1 )}\\
&= \frac{k_u(\bt)-1}{L(\bt)-1}\cdot
\end{align*}
\end{proof}

\subsection{Construction of the partition-valued process
  $\Pi_\text{GW}^{[n]}$}
\label{sec:GW-partition}

Let $\alpha\in [1/2,1)$.
Recall that the function $g$ defined by \reff{eq:def-g} is the generating
function of a probability measure $\nu_g$ on $\N$. We denote by
$G_g(dT)$ the distribution on $\T$ of the critical GW  tree
with offspring distribution $\nu_g$. We will denote 
by $\bP$ the probability measure on $\T\times [0,+\infty ]^\cu$:
\[
\bP(dT,d\xi)=G_g(dT)\bP^T(d\xi). 
\]

Under $\bP$, the random tree $T$ is a GW  tree whose 
offspring distribution $\nu_g$ has generating function $g$ given by
\reff{eq:def-g}.  According to Propositions 2.1 and 3.2 in \cite{adh:pcrtst},
$(\cp_\theta(T), \theta\geq 0)$ is a Markov process and $\cp_\theta(T)$
is  a GW  tree whose reproduction 
law has generating function $g_\theta$, with:
\[
g_\theta(r)=1+(1+\theta) \left[g\left(\frac{r}{1+\theta}\right)-
  g\left(\inv{1+\theta}\right)\right].
\]
Notice that:
\begin{equation}
   \label{eq:def-gq}
g_\theta(r)=r+\alpha \frac{(1-r+\theta)^{1/\alpha}
  -\theta^{1/\alpha}}{(1+\theta)^{(1/\alpha)-1}}\cdot
\end{equation}

For every positive integer $n$, we set:
\[
\bP_n(\bullet)=\bP( \bullet \bigm| L(T)=n).
\]
Under
$\bP_n$, the distribution of the tree $T$ is given by the following
formula (see \cite{dlg:rtlpsbp}, Theorem 3.3.3, or \cite{m:nfst}), for
$\bt\in \T_n$:
\begin{equation}\label{eq:proba-tree}
\bP_n(T=\bt)=n!\left(\prod _{v\in
    \cn(\bt)}\frac{p_{k_v(\bt)}}{k_v(\bt)!}\right)\frac{\alpha^{n-1}
  \Gamma(1-\alpha)}{\Gamma(n-\alpha)} 
\end{equation}
where $p_1=0$ and, for $k\geq 2$, $p_k=\val{(1-\gamma)(2-\gamma)\cdots
  (k-\gamma)}$.

Let  $n\in\N^*$.  Let  $T$  be  a random  tree  distributed as  $\bP_n$.
Conditionally   on   $T$,   we   define  a   uniform   random   labeling
$U_1,\ldots,U_n$ of  the leaves of  $T$, independently of  the variables
$(\xi_u,u\in T)$.  Recall  the set of ancestors defined in \reff{eq:Au}
and the pruning procedure $\cp_\theta$ introduced
in Section \ref{sec:d-tree-valued}.   We define the equivalence relation
$\calr_\theta^{[n]}$ on $\{1,2,\ldots, n\}$ by:
$i\calr_\theta^{[n]}  j$   if   $A_{U_i}\bigcap    A_{U_j}   \bigcap
\cl(\cp_\theta(T))$ is non empty, that is $U_i$ and $U_j$ have a leaf of
$\cp_\theta(T)$ as common ancestor.   Then, for every $\theta\ge 0$, let
$\hat \Pi_\text{GW}^{[n]}(\theta)$  be the equivalence  classes of the
equivalence  relation  $\calr_\theta^{[n]}$    of  the $n$  first
integers. Let $\Pi_\text{GW}^{[n]}=(\Pi_\text{GW}^{[n]}(k), k\in
\N)$    be    the    discrete    process    associated   with    $\hat
\Pi_\text{GW}^{[n]}= (\hat
\Pi_\text{GW}^{[n]}(\theta),  \theta\geq 0)$ until  it reaches  the absorbing
state  (which is  the trivial  partition  consisting in  one block)  and
afterward the discrete process remains constant.

We end this section with an elementary lemma which will be used in the
proof of Theorem~\ref{th:main}.

\begin{lem}
 \label{lem:k0}
  We have for $n\geq 2$:
\begin{equation}
   \label{eq:ln}
 \bE_n\left[k_\emptyset(T) -1\right]= \frac{1-\alpha}{\alpha} \, 
\frac{\Gamma\left(1-\alpha\right)}{\Gamma\left(\alpha\right)}\, 
\frac{\Gamma\left(n-1+\alpha\right)}{\Gamma\left(n-\alpha\right)}\cdot
\end{equation}
\end{lem}

\begin{proof}
 We consider the generating function of $(k_\emptyset(T), L(T))$ under
 $\bP$, that is   $H(s,t)=\bE\left[s^{k_{\emptyset}(T)}
    t^{L(T)}\right]$. 
Using  the branching
  property of GW  trees, we have:
\begin{equation}
   \label{eq:n_0,L}
H(s,t)
=\bE\left[s^{k_{\emptyset}(T)}
  \bE[t^{L(T)}]^{k_{\emptyset}(T)}\ind_{\{k_{\emptyset}(T)
    \neq    0\}}\right] + t\bP(k_{\emptyset}(T)=0).
\end{equation}
Notice that
$g(s)=\bE\left[s^{k_{\emptyset}(T)} \right]=H(s,1)$. We
set $h(t)=H(1,t)= \bE\left[ t^{L(T)}\right]$ 
 the generating function of $L(T)$. 
So that \reff{eq:n_0,L} becomes:
\begin{equation}
   \label{eq:Hgh}
H(s,t)=g(s\,h(t)) - g(0) (1-t).
\end{equation}
Taking $s=1$ in
\reff{eq:Hgh}, we get:
\begin{equation}
   \label{eq:gen-leaf}
g(h(t))-h(t)= g(0)(1-t). 
\end{equation}
Using expression \reff{eq:def-g}, we get:
\[
h(t)=1- (1-t)^\alpha \quad\text{and}\quad
H(s,t)=s\, h(t) + \alpha(1-s\, h(t))^{1/\alpha} - \alpha(1-t).
\]
We deduce that:
\begin{align*}
\bE\left[k_\emptyset(T) t^{L(T)}\right]
=\frac{\partial H}{\partial s}(1,t)
&=h(t)
-h(t)(1-h(t))^{(1/\alpha)-1}\\
&= \bE\left[t^{L(T)}\right] - \left[1- (1-t)^\alpha \right]
  (1-t)^{1-\alpha}\\
&= \bE\left[t^{L(T)}\right] - (1-t)^{1-\alpha}+1-t.
\end{align*}
This gives:
\[
\bE\left[(k_\emptyset(T)-1) t^{L(T)}\right]= - (1-t)^{1-\alpha}+1-t.
\]
For $n\geq 2$, we get:
\begin{align*}
\bE\left[(k_\emptyset(T)-1) \ind_{\{L(T)=n\}}\right]
&=\inv{n!} \left(\frac{d^n}{dt^n} \bE\left[(k_\emptyset(T)-1)
    t^{L(T)}\right]\right)_{|_{t=0}}\\ 
&=\inv{n!} \left(1-\alpha\right) \prod_{k=0}^{n-2}
\left(\alpha +k \right)\\
&=\inv{n!}\left(1-\alpha\right)
\frac{\Gamma\left(n-1+\alpha\right)}{\Gamma\left(\alpha\right)}\cdot
\end{align*}
We also get for $n\geq 2$:
\[
\bP(L(T)=n)
=\inv{n!} h^{(n)}(0)
=\inv{n!} \alpha  \prod_{k=1}^{n-1}
\left(k-\alpha  \right)
=\inv{n!}\alpha
\frac{\Gamma\left(n-\alpha\right)}{\Gamma\left(1-\alpha\right)}\cdot
\]
We deduce that:
\[
\bE_n\left[k_\emptyset(T)-1\right]
= \frac{\bE\left[(k_\emptyset(T)-1)
    \ind_{\{L(T)=n\}}\right]}{\bP(L(T)=n)}
= \frac{1-\alpha}{\alpha} \frac{\Gamma(1-\alpha)}{\Gamma(\alpha)}
\frac{\Gamma(n-1+\alpha)}{\Gamma(n-\alpha)}\cdot
\]
\end{proof}

\section{Proof of Theorem \ref{th:main}}
\label{sec:proof-1}

Let   $\alpha\in   [1/2,1)$    and   $\Lambda$   given   by
\reff{eq:def-Lambda}.   Notice  that  the  probability  that  the  first
coalescence   event  for   $\Pi^{[n]}_\text{dis}$  corresponds   to  the
collision  of  $k$  given  blocks is  $\lambda_{n,k}/\lambda_{n}$,  with
$\lambda_{n,k}$     and     $\lambda_n$     given    respectively     by
\reff{eq:lambda-bn} and \reff{eq:lambda-n}.

Theorem    \ref{th:main}   is    a   direct    consequence    of   Lemma
\ref{lem:rate-coal}  which states  that the  probability that  the first
coalescence event for $\Pi^{[n]}_\text{GW}$ corresponds to the collision
of  $k$  given  blocks  is  $\lambda_{n,k}/\lambda_{n}$,  and  of  Lemma
\ref{lem:Pn->Pk}, which  states that after the  first coalescence event,
the law of the pruned tree under $\bP_n$ conditionally given that it has
$k$ leaves is exactly $\bP_k$.

The proof of Lemme \ref{lem:rate-coal} (resp. \ref{lem:Pn->Pk}) is given
in Section \ref{sec:rate} (resp. \ref{sec:last}). 

\subsection{Computation of the coalescence rates}
\label{sec:rate}

We first give an intermediate lemma. 
For  $\alpha\in (0,1)$ and $\lambda>\alpha-1$, we set: 
\begin{equation}
   \label{eq:def-phi}
 \phi_{1+\alpha,1-\alpha}(\lambda)
=\int_0^1 \left(1- (1-x)^\lambda \right)
x^{\alpha-2}(1-x)^{-\alpha} \; dx.
\end{equation}

\begin{lem}
   \label{lem:f(l)}
For  $\alpha\in (0,1)$ and $\lambda> \alpha-1$, we have:
\begin{equation}
   \label{eq:f(l)-2}
 \phi_{1+\alpha,1-\alpha}(\lambda)=\lambda \frac{\Gamma(\alpha)
   \Gamma(\lambda+1-\alpha)}{(1-\alpha) \Gamma(\lambda+1)}\cdot 
\end{equation}
\end{lem}
Notice that for $\lambda>0$, \reff{eq:f(l)-2} reduces to:
\begin{equation}
   \label{eq:f(l)}
 \phi_{1+\alpha,1-\alpha}(\lambda)= \frac{\Gamma(\alpha)
   \Gamma(\lambda+1-\alpha)}{(1-\alpha) \Gamma(\lambda)}\cdot 
\end{equation}

\begin{proof}
We set:
\[
I=  \int_0^1 \left((1-u)^ {-\alpha}-1 \right) \; u^{\alpha-2}\, du. 
\]
Notice that $I$ is finite and $ \phi_{1+\alpha,1-\alpha}(\alpha)=I$. 
For $\lambda>\alpha$, using an integration by part, we have:
\begin{align*}
   \phi_{1+\alpha,1-\alpha}(\lambda)
&=\int_0^1 \left(1- (1-x)^\lambda \right)
x^{\alpha-2}(1-x)^{-\alpha} \; dx\\
&=\int_0^1 \left((1-x)^{-\alpha}-1 \right)
x^{\alpha-2}\; dx + \int_0^1 \left(1- (1-x)^{\lambda-\alpha} \right)
x^{\alpha-2} dx\\
&=I- \inv{1-\alpha} + \frac{\lambda-\alpha}{1-\alpha}
\int_0^1 (1-x)^{\lambda-\alpha-1}x^{\alpha-1}\; dx\\
&= I- \inv{1-\alpha} + \frac{\Gamma(\alpha)
  \Gamma(\lambda+1-\alpha)}{(1-\alpha) \Gamma(\lambda)}\cdot
\end{align*}
We now compute $I$. For $\lambda=1$, we also have:
\[
   \phi_{1+\alpha,1-\alpha}(1)
=\int_0^1 
x^{\alpha-1}(1-x)^{-\alpha} \; dx 
= \Gamma(\alpha)\Gamma(1-\alpha) .
\]
We deduce that:
\[
I- \inv{1-\alpha} + \frac{\Gamma(\alpha)
  \Gamma(2-\alpha)}{(1-\alpha) \Gamma(1)}=
\phi_{1+\alpha,1-\alpha}(1)
= \Gamma(\alpha)\Gamma(1-\alpha) .
\]
This readily implies that $I=1/(1-\alpha)$ and thus \reff{eq:f(l)-2}
holds for $\lambda\geq \alpha$. Then uses that the right-hand sides of
\reff{eq:def-phi} and \reff{eq:f(l)-2} are analytic for
$\lambda>\alpha-1$ to get that \reff{eq:f(l)-2} also holds for
$\lambda>\alpha-1$. 
\end{proof}

Recall that $\lambda_{n,k}$
and $\lambda_n$ are given respectively by \reff{eq:lambda-bn}
and \reff{eq:lambda-n}, for $\Lambda$ given by  \reff{eq:def-Lambda}.

\begin{lem}
\label{lem:lnk/ln}
Let $\alpha\in [1/2,1)$. We have for $2\leq k\leq n$:
\begin{equation}\label{eq:lambda}
\frac{\lambda_{n,k}}{\lambda_n}=\frac{1-\alpha}{\Gamma(\alpha+1)}
\frac{\Gamma(k+\alpha-1)\Gamma(n-k-\alpha+1)}{\Gamma(n-\alpha)}
\frac{1}{n-1}\cdot   
\end{equation}
\end{lem}

\begin{proof}
We have
\begin{align*}
\lambda_{n,k} & =\int_0^1u^{k-2}(1-u)^{n-k}\Lambda (du)\\
 & =\int_0^1u^{k-2+\alpha}(1-u)^{n-k-\alpha}du\\
& =\beta(k+\alpha-1,n-k-\alpha+1)\\
& =\frac{\Gamma (k+\alpha-1)\Gamma(n-k-\alpha+1)}{\Gamma(n)},
\end{align*}
and 
\[
\lambda_n=\sum_{k=2}^n\left(\begin{array}{c}n\\k\end{array}\right)
\lambda_{n,k}=
\int_0^1(1-(1-u)^n-nu(1-u)^{n-1})u^{-2}\Lambda(du).
\]
Then using notations  \reff{eq:def-phi} and 
\reff{eq:f(l)}, we deduce that:
\begin{align*}
   \lambda_n
&= \phi_{1+\alpha,1-\alpha}(n)
- n \int_0^1 u^{\alpha-1} (1-u)^ {n-1-\alpha} \, du\\
&= \frac{\Gamma(\alpha)
   \Gamma(n+1-\alpha)}{(1-\alpha) \Gamma(n)}- n
 \frac{\Gamma(\alpha)\Gamma(n-\alpha)}{\Gamma(n)} \\
&= \left(\frac{n-\alpha}{1-\alpha} - n \right)
\frac{\Gamma(\alpha)\Gamma(n-\alpha)}{\Gamma(n)} \\
&=  (n-1)\frac{\alpha}{1-\alpha}
\frac{\Gamma(\alpha)\Gamma(n-\alpha)}{\Gamma(n)}\cdot 
\end{align*}
The expression obtained for $\lambda_{n,k}$ then gives the result.
\end{proof}

If $\bt_1$ and $\bt_2$ are two discrete trees and $u\in \cl(\bt_1)$ is a
leaf of $\bt_1$, we shall  denote by $\bt_1\circledast_u \bt_2$ the tree
obtained by grafting  the tree $\bt_2$ on the leaf  $u$ of $\bt_1$, that
is:
\begin{equation}
   \label{eq:tut}
\bt_1\circledast_u \bt_2=\bt_1\cup\{uv,\ v\in \bt_2\}.
\end{equation}

\begin{lem}
\label{lem:rate-coal}
Let $\alpha\in  [1/2,1)$. The probability  under $\bP_n$ that  the first
coalescence event  in $\Pi_\text{GW}^{[n]} $  is the coalescence  of $k$
given integers into one block is $\lambda_{n,k}/\lambda_n$.
\end{lem}

\begin{proof}
Let $A_k$ be the event that the first coalescence event corresponds 
to the $k$ first integers merging together. By exchangeability, the lemma
is proved as soon as we check  that $\bP_n(A_k)=\lambda_{n,k}/\lambda_n$.

The event $A_k$ is  realized, if and only if:
\begin{itemize}
\item The initial tree $T$ is of the form $\bt_1\circledast_u \bt_2$ for
  some $\bt_2\in \T_k$ and $\bt_1\in \T_{n-k+1}$ and $u\in \cl(\bt_1)$. 
\item The leaves of $\bt_2$ are labeled from $1$ to $k$ (and
  therefore, the leaves of $\bt_1$ except $u$ are labeled from $k+1$ to
  $n$). This occurs with probability $\frac{k!(n-k)!}{n!}$.
\item The first  chosen node of $\bt_1\circledast_u \bt_2$  is $u$. This
  occurs    according   to    Lemma   \ref{lem:Ax}    with   probability
  $\frac{k_\emptyset(\bt_2)-1}{n-1}$.
\end{itemize}

Thus, using \reff{eq:proba-tree} for the probability of
having a given tree, we have:
\begin{align*}
\bP_n(A_k) 
& =\sum_{\overset{\bt_1\in\T_{n-k+1}}{\overset{\bt_2\in \T_k}{u\in
  \cl(\bt_1)}}}\bP_n(T=\bt_1\circledast_u \bt_2)\frac{k!(n-k)!}{n!}
\frac{k_\emptyset(\bt_2)-1}{n-1}\\ 
& =\sum_{\overset{\bt_1\in\T_{n-k+1}}{\overset{\bt_2\in \T_k}{u\in
  \cl(\bt_1)}}} n!\left(\prod_{v\in\cn(\bt_1\circledast_u
    \bt_2)}\frac{p_{k_v(\bt_1\circledast_u
    \bt_2)}}{k_v(\bt_1\circledast_u \bt_2)!}\right)\frac{\alpha^{n-1}\Gamma(1-\alpha)}{\Gamma(n-\alpha)}\frac{k!(n-k)!}{n!}
  \frac{k_\emptyset(\bt_2)-1}{n-1}\\
&= (n-k+1)\sum_{\overset{\bt_1\in\T_{n-k+1}}{\bt_2\in \T_k}}
  \frac{n!}{k!(n-k+1)!}\bP_{n-k+1}(T=\bt_1)\bP_k(T=\bt_2)\\
& \qquad\qquad
{\frac{\alpha^{n-1}\Gamma(1-\alpha)}{\Gamma(n-\alpha)}\frac{\Gamma(n-k-\alpha+1)}{\alpha^{n-k}\Gamma(1-\alpha)}\frac{\Gamma(k-\alpha)}{\alpha^{k-1}\Gamma(1-\alpha)}\frac{k!(n-k)!}{n!}
  \frac{k_\emptyset(\bt_2)-1}{n-1}}\\
& =\frac{\Gamma(n-k-\alpha+1)\Gamma(k-\alpha)}{\Gamma(n-\alpha)\Gamma(1-\alpha)}\frac{1}{n-1}\bE_k\left[k_\emptyset(T)-1\right].
\end{align*}
We then use Lemma \ref{lem:k0} and Lemma \ref{lem:lnk/ln} to conclude. 
\end{proof}

\subsection{Law of the tree after the first coalescence event}
\label{sec:last}
Let $S$ be the time of the first coalescence event and recall that
$\cp_S(T)$ denote the pruned tree at the first coalescence
event.

\begin{lem}
   \label{lem:Pn->Pk}
Let $\bt\in \T_k$. We have:
\begin{equation}
   \label{eq:PSn->k}
\bP_n(\cp_S(T)=\bt\bigm| L(\cp_S(T))=k)=\bP_k(T=\bt).
\end{equation}
\end{lem}
\begin{proof}
Let $\bt\in \T_k$. We obtain $\bt$ just after the first coalescence event
if $T$ is of the form $\bt\circledast_u\bs $ for some $s\in \T_
{n-k+1}$, $u\in \cl(\bt)$ and $u$ is the first chosen internal
node. This gives:
\begin{align*}
\bP_n(\cp_S(T)=\bt) & = \sum_{\overset{u\in \cl(\bt)}{\bs\in
    \T_{n-k+1}}}\bP_n(T=\bt\circledast_u\bs)\frac{k_\emptyset(\bs)-1}{n-1}\\
& =k \sum_{\bs\in \T_{n-k+1}}n!
\left(\prod_{v\in\cn(\bt)}\frac{p_{k_v(\bt)}}{k_v(\bt)!}
\prod_{v\in\cn(\bs)}\frac{p_{k_v(\bs)}}{k_v(\bs)!}
\right)\frac{\alpha^{n-1}\Gamma(1-\alpha)}{\Gamma
    (n-\alpha)}\frac{k_\emptyset(\bs)-1}{n-1}\\
& =k \sum_{\bs\in
  \T_{n-k+1}}\frac{n!}{k!(n-k+1)!}\bP_k(T=\bt)\bP_{n-k+1}(T=\bs)\\
&\qquad\qquad\frac{\alpha^{n-1}\Gamma(1-\alpha)}{\Gamma
    (n-\alpha)}\frac{\Gamma(k-\alpha)}{\alpha^{k-1}\Gamma(1-\alpha)}\frac{\Gamma(n-k+1-\alpha)}{\alpha
    ^{n-k}\Gamma(1-\alpha)}\frac{k_\emptyset(\bs)-1}{n-1}\\
&
=\frac{n!}{(k-1)!(n-k+1)!}\frac{\Gamma(n-k+1-\alpha)\Gamma(k-\alpha)}{\Gamma(n-\alpha)\Gamma(1-\alpha)}\\
& \qquad\qquad\qquad\qquad\inv{n-1}\bE_{n-k+1}[k_\emptyset(T)-1]\bP_k(T=\bt).
\end{align*}
As the term in front of  $\bP_k(T=\bt)$ does not depend on $\bt$, it has
to  be equal  to $\bP_n(L(\cp_S(T))=k)$  and  therefore \reff{eq:PSn->k}
holds.
\end{proof}

\section{Pruning of rooted real trees and proof of Theorem
  \ref{th:main2} }
\label{sec:proof-2}
The aim of this section is to use the pruning
procedure for  Lévy trees developed in \cite{ad:falpus} to give a consistent
representation of the family of coalescent processes $(\hat
\Pi^{[n]}_\text{GW}, n\in \N^*)$, see Corollary \ref{cor:GW-CRT} and
thus deduce Theorem \ref{th:main2}.

\subsection{The CRT framework}

\subsubsection{Real trees}

Real trees have been introduced first in the field of geometric group
theory (see for instance \cite{dmt:tto}) and then used later for
defining continuum random trees (the framework first appeared in
\cite{epw:rprtrgr}). A real tree is a metric space $(\ct,d)$ satisfying the following two
properties for every $x,y\in\ct$:
\begin{itemize}
\item (unique geodesic) There is a unique isometric map $f_{x,y}$ from
  $[0,d(x,y)]$ into $\ct$ such that $f_{x,y}(0)=x$ and
  $f_{x,y}(d(x,y))=y$.
\item (no loop) If $\varphi$ is a continuous injective map from
  $[0,1]$ into $\ct$ such that $\varphi(0)=x$ and $\varphi(1)=y$, then
$$\varphi([0,1])=f_{x,y}([0,d(x,y)]).$$
\end{itemize}
A rooted real tree is a real tree with a distinguished vertex denoted
$\emptyset$ and called the root.

For every $x,y\in\ct$, we denote by $\lb x,y\rb$ the range of the map
$f_{x,y}$ (i.e. the only path in the tree that links $x$ to $y$) and we set
$\lb x,y\lb=\lb x,y\rb\setminus\{y\}$.

If $\ct$ is  a rooted real tree, for $x\in\ct$, we  define the degree of
$x$,  denoted  by  $n_x$,  as  the number  of  connected  components  of
$\ct\setminus\{x\}$.  The leaves of $\ct$ is $\cl(\ct)=\{x\in
\ct\setminus \{\emptyset\}; \, n_x=1\}$. 
If $n_x\ge 3$, we say that $x$ is a branching point of $\ct$.  We denote
by $\cb_\text{br}(\ct)$ the set of branching points of $\ct$. The height
of $\ct$ is $H_\text{max}(\ct)=\sup\{d(\emptyset,x); \; x\in \ct\}$.
Let $(x_i, i\in I)$ be a family of elements of $\ct$, we define their
most recent common ancestor denoted by $\text{MRCA}(x_i, i\in I)$ as the
element $x$ of $\ct$ such that $\lb \emptyset,x\rb=\bigcap _{i\in I} \lb
\emptyset,x_i\rb$.

A weighted rooted real tree $(\ct, d, \bm)$ is a rooted real tree
$(\ct, d)$ endowed with a $\sigma$-finite measure $\bm$ called the
mass measure.

\subsubsection{Stable L\'evy tree}

Set $\psi(\lambda)=\lambda^\gamma$  with $\gamma\in(1,2)$.  We  refer to
\cite{dlg:pfalt} and  \cite{adh:etiltvp} for the existence  of a measure
$\N[d\ct]$ on the set of  weighted locally compact rooted real tree such
that $\ct$ is  under $\N[d\ct]$ a Lévy tree  associated with the branching
mechanism $\psi$. For  the Lévy tree $(\ct, d,  \bm)$, $\N[d\ct]$ -a.e.,
the mass measure  has support $\cl(\ct)$ and has  no atom.
Furthermore, $\N[d\ct]$-a.e., all the branching points of the tree are
of infinite degree. Following  \cite{dlg:pfalt},
there exists  a local  time process $(\ell^a,  a\geq 0)$ with  values on
finite  measures on  $\ct$, which  is càdlàg  for the  weak  topology on
finite measures on $\ct$ and such that $\N^\psi[d\ct]$-a.e.:
\[
\bm(dx) = \int_0^\infty \ell^a(dx) \, da,
\]
$\ell^0=0$,  $\inf\{a > 0  ; \ell^a  = 0\}=\sup\{a  \geq 0  ; \ell^a\neq
0\}=H_{\text{max}}(\ct)$    and    for    every    fixed    $a\ge    0$,
$\N^\psi[d\ct]$-a.e. the  measure $\ell^a$ is supported  on $\{x\in \ct;
d(\emptyset,x)=a\}$  and  the   real  valued  process  $(\langle\ell^a,1
\rangle  , a\geq  0)$ is  distributed as  a continuous  state branching
process  (CSBP)  with branching  mechanism  $\psi$  under its  canonical
measure. In particular, as the total  size of a critical CSBP is finite,
we get that $\N$-a.e. $\sigma=\bm(\ct)$ is finite.

The set $\{  d(\emptyset,x),\ x\in \mathrm{Br}(\ct) \}$ coincides
$\N^\psi$-a.e.  with  the set  of  discontinuity  times  of the  mapping
$a\mapsto   \ell^a$.   Moreover,    $\N^\psi$-a.e.,   for   every   such
discontinuity    time     $b$,    there    is     a    unique    $x\in
\cb_\text{br}(\ct)$    such     that    $d(\emptyset,x)=b$    and
$\Delta_x>0$, such that:
\[ 
\ell^b = \ell^{b-} + \Delta_x \delta_{x}, 
\]
where $\Delta_x>0$ is called the mass of the node $x$. 
Intuitively $\Delta_x$ represents the size of the progeny of $x$. 

The scaling property of the stable Lévy tree implies that there exists a well
defined 
probability measure $\N^{(1)}$ defined as the measure $\N$ conditioned
on $\{\sigma=1\}$. The probability measure is also referred as the
normalized excursion measure for Lévy trees. 

\subsection{The partition-valued process}
Set $\psi(\lambda)=\lambda^\gamma$  with $\gamma\in(1,2)$.
\subsubsection{Pruning of the stable L\'evy tree}
We consider  the pruning procedure introduced  in \cite{ad:falpus} (this
procedure is defined when there is  no Brownian part in the Lévy process
with  index given  by the  branching  mechanism $\psi$).  Under $\N$  or
$\N^{(1)}$, conditionally  given $\ct$, we consider  a family $(E_x,x\in
\cb_\text{br}(\ct))$ of independent real  random variables such that the
random  variable  $E_x$  is  exponentially  distributed  with  parameter
$\Delta_x$.   This random  variable  represents the  time  at which  the
branching point $x$ is marked.  For every $\theta>0$, we set
\[
\ct_\theta=\{x\in\ct,\ \forall y\in\lb\emptyset ,
x\lb,\ E_y\ge\theta\}.
\]
The set $\ct_\theta$ is still a real tree which represents the tree $\ct$
  pruned at time $\theta$: we cut $\ct$ at  the points that are marked before
  time $\theta$ and keep the connected component of the tree that
  contains the root. We set $\ct_0=\ct$. 
By \cite{ad:falpus}, Theorem 1.5, the tree $\ct_\theta$ is distributed
under $\N$ as a L\'evy tree with branching mechanism $\psi_\theta$
defined by:
\[
\psi_\theta(\lambda)=\psi(\lambda+\theta)-\psi(\theta).
\]
Moreover, by  \cite{ad:ctvmp}, the process $(\ct_\theta,\theta\ge 0)$
is under $\N$ a Markov process.

\subsubsection{Definition of the partition-valued process}
\label{sec:partition-L}
Under $\N$ or $\N^{(1)}$, conditionally on $\ct$, let $(F_i, i\in \N^*)$
be independent  random variables on  $\ct$ distributed according  to the
probability  mass measure  $\bm/\bm(\ct)$, and  independent of  the marks
$(E_x,   x\in   \cb_\text{br}(\ct))$.    Notice   that   $\N$-a.e.    or
$\N^{(1)}$-a.s. $(F_i, i\in \N^*)$ are leaves of $\ct$.  For $\theta\geq
0$,  we define the  equivalence relation  $\calr_\theta^\text{Lévy}$
on $\N^*$ by:
$i\calr_\theta^\text{Lévy}  j$  if  $\lb  \emptyset,F_i\rb  \bigcap  \lb
\emptyset,F_j\rb \bigcap  \cl(\ct_\theta)$ is  non empty, that  is $F_i$
and $F_j$ have  a leaf of $\ct_\theta$ as common  ancestor. This is very
close    to    the    definition    of    the    equivalence    relation
$\calr_\theta^\text{[n]}$ defined in Section \ref{sec:GW-partition}.  We
denote by  $\Pi_\text{Lévy}(\theta)$ the  partition of $\N^*$  formed by
the   equivalence  classes  of   $\calr_\theta^\text{Lévy}  $   and  set
$\Pi_\text{Lévy}=(\Pi_\text{Lévy}(\theta), \theta\geq 0)$.

\subsection{L\'evy sub-trees}

\subsubsection{Skeleton of finite real tree}
Let $\hat \bt$ be a real tree  with finite height and a finite number of
leaves, such  that the  leaves $(f_i, i\in  I(\hat \bt))$ are indexed
by a totally ordered set $I(\hat \bt)$. We define the skeleton $\tilde
\bt$ of the tree $\hat\bt$ as the discrete tree (belonging to $\T$) where we
forget the edge lengths. As the trees in $\T$ are ordered, we must be
a bit more rigorous for the definition of $\tilde\bt$.

 The skeleton $\tilde \bt$ of the real tree
with ordered leaves $\left(\hat  \bt, (f_i, i\in I(\hat \bt))\right)$ is
defined recursively as follows.   We define $k_\emptyset(\tilde \bt)$ as
the degree of  $\text{MRCA}(f_i, i\in I(\hat \bt))$ the  ancestor of all
the leaves of $\hat  \bt$.  If $k_\emptyset(\tilde \bt)=0$, then $\tilde
\bt$ is  reduced to $\emptyset$. In  this case $\hat \bt$  has one leaf,
let $f$  be its label, and the  discrete tree $\tilde \bt$  has thus one
leaf to which we give  the label $f$. If $k_\emptyset(\tilde \bt)>0$, then
we consider  the $k_\emptyset(\tilde \bt)$ connected  components of $\hat
\bt\setminus\{  \text{MRCA}(f_i,  i\in  I(\hat  \bt))\} $  that  do  not
contain  the root  and label  them from  1 to  $k_\emptyset(\tilde \bt)$
according to the lowest label of  the leaves of $\hat \bt$ which belongs
to them.  This  gives an ordered family $(\hat  \bt_k, k\in \{1, \ldots,
k_\emptyset(\tilde \bt)  \})$ of  real trees, and  let $\text{MRCA}(f_i,
i\in I(\hat \bt))\})$  be the root of each one.   For $k\in \{1, \ldots,
k_\emptyset(\tilde  \bt)  \}$, let  $I(\hat  \bt_k)=\{i\in I(\hat  \bt);
f_i\in \hat  \bt_k)$ be the labels  of the leaves  of $\hat \bt_k$  and the
discrete tree $\tilde \bt_k$ is the skeleton of $\left(\hat \bt_k, (f_i,
  i\in I(\hat \bt_k))\right)$.

Notice  that $\tilde  \bt$ is  finite, $k_u(\tilde  \bt)\neq 1$  for all
$u\in \tilde \bt$, and $\hat \bt$  and $\tilde \bt$ have the same number
of leaves.  In  the previous construction to a leaf  $f_i$ of $\hat \bt$
with  label $i$ corresponds  a unique  leaf $e_i$  of $\tilde  \bt$ with
label $i$.  For $u\in \tilde \bt$, we define $\tilde\bt_u$ the
sub-tree of $\tilde\bt$ attached to the node $u$ i.e.
\begin{equation}\label{eq:def-tu}
\tilde\bt_u=\{w\in\cu,\ uw\in \tilde\bt\},
\end{equation}
and let $I_u=\{i; e_i\in \tilde \bt_u\}$.
Define $\hat \bt_u$ as $\bs_u=\hat \bt \backslash \bigcup_{i\not\in I_u}
\lb    \emptyset,    f_i\rb$    to     which    we    add    the    root
$\emptyset_u=\overline{\bs_u} \backslash\bs_u$, and $I(\hat \bt_u) =\{i;
e_i\in \tilde  \bt_u\}$.  Notice that by construction  $\tilde \bt_u$ is
the  skeleton   of  $\left(\hat  \bt_u,  (f_i;  i\in   I(\hat  \bt_u)  )
\right)$. We  say that  $u\in \tilde \bt$  are the individuals  of $\hat
\bt$, and define their lifetime as the length $h_u$ of the geodesic $B(u)=\lb
\emptyset_u,\text{MRCA}(f_i, i\in I(\hat \bt_u)) \rb$. We say the
corresponding node in $\hat \bt$ of $u\in \tilde \bt$ is
$C(u)=\text{MRCA}(f_i, i\in I(\hat \bt_u))$. 

Notice it is easy to reconstruct
$\hat \bt$ from $\tilde \bt $ and the family of lifetime $(h_u, u\in
\tilde \bt)$.

\subsubsection{Coalescence of Lévy tree and GW tree}
Let $M$ be,  under $\N$ or $\N^{(1)}$ conditionally  on $\ct$, a Poisson
random variable  with finite mean $\sigma=\bm(\ct)$. We shall work on $\{M\geq 1\}$.  On $\{M\geq 1\}$,
let $\hat  T_0$ be the real sub-tree  of $\ct$ generated by  the root and
$(F_i, 1\leq i\leq M)$:
\[
\hat T_0 = \bigcup_{1\le i\le M}\lb \emptyset, F_i\rb. 
\]
Since  $\bm$ has  support $\cl(\ct)$  and has  no atom,  we  deduce that
$(F_i, 1\leq i\leq  M)$ are distinct and are the leaves  of $\hat T_0$.  

We denote  by $\tilde T_0$ the  skeleton of $\hat T_0$  with the labeled
leaves $(F_i, 1\leq i\leq M)$.  According to \cite{dlg:rtlpsbp}, Theorem
3.2.1, the  tree $\hat T_0$ is distributed  under $\N[\; \cdot\bigm|M\ge
1]$ as a continuous GW tree (i.e. a GW tree with edge-lengths) such that
\begin{itemize}
\item The discrete tree $\tilde T_0$ is a GW  tree with
  offspring distribution characterized by its generating function $g$ defined
  by  \reff{eq:def-g} with $\alpha=1/\gamma$. 
\item Lifetimes of individuals  $(h_u, u\in \tilde T_0)$ are independent
  random   variables  with   exponential  distribution   with  parameter
  $\gamma$.
\end{itemize}

We must first prove the following lemma which will be a key point in
the sequel. Its proof relies on the scaling property of the Lévy tree.
\begin{lem}
The
distributions of $\hat T_0$ under $\N[\; \cdot\bigm| M=n]$ and
under $\N^{(1)}[\; \cdot \bigm| M=n]$ are the same.
\end{lem}

\begin{proof}
For a tree $\ct$ and points $x_1,\ldots,x_n$ of $\ct$, let us denote
by $T(\ct,x_1,\ldots,x_n)$ the tree spanned by the points $(x_i)$ and
the root of the tree and $\tilde T(\ct,x_1,\ldots,x_n)$ the associated
discrete tree so that under $\N[\; \cdot\bigm| M=n]$ or $\N^{(1)}[\;
  \cdot \bigm| M=n]$, we have
$$\tilde T_0=\tilde T(\ct,F_1,\ldots,F_n).$$

Then, for every bounded measurable function $\phi$, we have
$$\N\left[\phi\bigl(\tilde
  T(\ct,F_1,\ldots,F_n)\bigr)\ind_{\{M=n\}}\right]=\N\left[\phi\bigl(\tilde
  T(\ct,F_1,\ldots,F_n)\bigr)\frac{\sigma^n}{n!}\expp{-\sigma}\right].$$

Let $\nu$ be the distribution of $\sigma$ under $\N$ i.e. the only
measure $\nu$ such that for every $\lambda>0$,
$$\int_0^{+\infty}(1-\expp{-\lambda u})\nu(du)=\lambda^\alpha.$$

Then we have
$$\N\left[\phi\bigl(\tilde
  T(\ct,F_1,\ldots,F_n)\bigr)\ind_{\{M=n\}}\right]=\int_0^{+\infty}\N^{(u)}\left[\phi\bigl(\tilde
  T(\ct,F_1,\ldots,F_n)\bigr)\right]\frac{u^n}{n!}\expp{-u}\nu(du).$$
Using the scaling property of the stable L\'evy tree (see
\cite{dlg:rtlpsbp} Section 3.3), we have that the law of the tree
$\ct$ under $\N^{(u)}$ is the same as the law of $u^{1-\alpha}\ct$
under $\N^{(1)}$ where the notation $\lambda \ct$ means that we
  multiply the distance that defines $\ct$ by the factor $\lambda$
  (i.e. we scale all the edge lengths by $\lambda$). Moreover, as we
  only look at discrete trees, this factor does not modify the tree
  $\tilde T_0$. Therefore, we get:
\begin{align*}
\N\left[\phi\bigl(\tilde
  T(\ct,F_1,\ldots,F_n)\bigr)\ind_{\{M=n\}}\right] & =\int_0^{+\infty}\N^{(1)}\left[\phi\bigl(\tilde
  T(\ct,F_1,\ldots,F_n)\bigr)\right]\frac{u^n}{n!}\expp{-u}\nu(du)\\
& =\N^{(1)}\left[\phi\bigl(\tilde
  T(\ct,F_1,\ldots,F_n)\bigr)\right]\N[M=n].
\end{align*}
We deduce:
\begin{multline*}
\N[\phi(\tilde T_0)\bigm|M=n]=\N\left[\phi\bigl(\tilde
  T(\ct,F_1,\ldots,F_n)\bigr)\bigm|M=n\right]\\
=\N^{(1)}\left[\phi\bigl(\tilde
  T(\ct,F_1,\ldots,F_n)\bigr)\right]=\N^{(1)}[\phi(\tilde
  T_0)\bigm|M=n]
\end{multline*}
since $\ct$ and $M$ are independent under $\N^{(1)}$.
\end{proof}

We now consider the marks that define the pruned tree $\ct_\theta$ and
we define on the event $\{M\ge 1\}$  the tree $\hat T_\theta$ as the tree $\hat T_0$ pruned on the
same marks, in other words, we set
\[
\hat T_\theta=\hat T_0\cap\ct_\theta.
\]
Let    $\hat    \Pi_\text{Lévy}^{[n]}$    be    the   restriction    of
$\Pi_\text{Lévy}$  to  the  $n$   first  integers.  By  construction,  if
$C_\theta$ is  an element of $\hat   \Pi_\text{Lévy}^{[n]}(\theta)$, then there
exists  a leaf  $x$ of  $\hat  T_\theta$ such  that $x$  belongs to  the
sub-tree $\bigcup _{i\in C_\theta} \lb \emptyset, F_i\rb$, and $x$ is the
only leaf of  $\hat T_\theta$ with this property.  We set $C_\theta$ for
the label of  $x$, and we consider the order of  the elements of $\tilde
\Pi_\text{Lévy}^{[n]}$ given by the  order of their smallest integer. We
set $I_\theta=I(\hat T_\theta)  $ for the labels of  the leaves of $\hat
T_\theta$ and $(F^\theta_i, i\in I_\theta)$ for the leaves of $\hat
T_\theta$. 

We denote by $\tilde T_\theta$  the skeleton of $\hat T_\theta$
with the labeled leaves $(F_i^\theta, i\in I_\theta)$. 
According to \cite{adh:pcrtst}, Proposition 4.1, the tree
$\hat T_\theta$ is distributed under $\N[\; \cdot\bigm|M\ge 1]$ as a continuous
GW  tree such that
\begin{itemize}
\item $\tilde T_\theta$ is a GW tree with offspring distribution
  characterized by its 
  generating function $g_\theta$ given in 
  \reff{eq:def-gq} with $\alpha=1/\gamma$. 
\item The lifetimes of individuals $(h_u, u\in \hat T_\theta)$ are
  independent  random variable with exponential  distribution
  with parameter $\psi_\theta'(1)=\gamma(1+\theta)^{\gamma-1}$.
\end{itemize}

The following Lemma is a consequence of Theorem 6.1 of
\cite{adh:pcrtst}.

\begin{lem}
\label{lem:discrete_pruning}
The process $(\tilde T_\theta,\theta\ge 0)$ is distributed
under $\N[\; \cdot\bigm|M\ge 1]$ as the process $(\cp_\theta(T),\theta\ge
0)$ under $\bP$. 
\end{lem}

\begin{proof}
Let $\theta>0$. 
Theorem 6.1 of \cite{adh:pcrtst}  describes how
$\hat T_\theta$ is obtained from $\hat T_0$: 
\begin{itemize}
\item  A branching  point $x$  of  $\hat T_0$  with $k_x=k_x(\hat  T_0)$
  children is marked at time $\tau_x$ with distribution given by:
\[
\N[\tau_x\ge \theta\bigm|\hat T_0]
=-\int_\theta^{+\infty}\frac{\psi^{(k_x+1)}(1+z)}{\psi^{(k_x)}(1)}dz
 =\frac{\psi^{(k_x)}(1+\theta)}{\psi^{(k_x)}(1)}
 =\left(\frac{1}{1+\theta}\right)^{k_x-\gamma}.
\]
\item A branch $B$ of length $h$ is marked at time $\tau_B$ with
  distribution given by:
\[
\N[\tau_{B}\ge \theta\bigm|\hat T_0]
 =\exp\left(-h\int_0^\theta \psi''(1+z)dz\right)
 =\expp{-\bigl(\psi'(1+\theta)-\psi'(1)\bigr)h}.
\]
\end{itemize}
Then the tree  $\hat T_0$ is cut according to the  marks present at time
$\theta$ and  the tree $\hat  T_\theta$ is the connected  component that
contains the  root.  Therefore, the  tree $\tilde T_\theta$  is obtained
from the  tree $\tilde T_0$  by a pruning  at node. A node  $u\in \tilde
T_0$ is  marked if the corresponding  node $C(u)\in \hat T_0$  is marked at
time  $\theta$ in the  previous procedure  OR the  branch $B(u)$ 
with length $h_u$ is marked. So the node $u$ of $\tilde T_0$ is marked
at time $\zeta_u=\tau_{C(u)}\wedge \tau_{B(u)}$  and using that the edge lengths
of  $\hat  T_0$  are  independent  and  exponentially  distributed  with
parameter $\gamma=\psi'(1)$, we have with $k_u=k_u(\hat T_0)$:
\begin{align*}
\N[\zeta_u\ge \theta\bigm|\tilde T_0] 
& = \N[\tau_{C(u)}\ge \theta\bigm| \tilde T_0]\; \N[\tau_{B(u)}\ge
\theta\bigm| \tilde T_0] \\
&= \left( \inv{1+\theta}\right)^{k_u-\gamma}
\int_0^{+\infty}dh\,
\gamma\expp{-\gamma h}  \expp{-\bigl(\psi'(1+\theta)-\gamma \bigr)h}\\
& =\left( \inv{1+\theta}\right)^{k_u-\gamma}
\left(\inv{1+\theta}\right)^{\gamma-1}\\ 
& =\left(\frac{1}{1+\theta}\right)^{k_u-1}\cdot
\end{align*}
Since the cutting time $\tau_{C(u)}$ and $\tau_{B(u)}$ are independent for all
internal  nodes $u$,  we  recover the  discrete  pruning procedure  that
defines the process $(\cp_\theta(T),\theta\ge 0)$ under $\bP$. To
conclude notice that $\tilde T_0$ and $T$ are GW tree with offspring
distribution characterized by its generating function $g$. 
\end{proof}

\subsection{Proof of Theorem \ref{th:main2}}

The  next  corollary  states  that  the  pruning  procedure  for  stable
GW   tree developed  in \cite{adh:pgwttvmp}  and  the pruning
procedure for  Lévy trees developed in \cite{ad:falpus}  and applied in
\cite{adh:pcrtst} to sub-trees with finite number of leaves coincide.

\begin{cor}
\label{cor:pruning}
Let $n\in\N$. The process $(\tilde T_\theta,\theta\ge 0)$ is distributed
under $\N[\; \cdot\bigm|M=n]$ as the process $(\cp_\theta(T),\theta\ge
0)$ under $\bP_n$. 
\end{cor}

\begin{proof}
   This is a direct consequence of Lemma
   \ref{lem:discrete_pruning} and the fact that $M=L(\tilde T_0)$. 
\end{proof}

Theorem \ref{th:main2}  follows directly from  Theorem \ref{th:main} and
from the following  corollary, which is a direct  consequence of Corollary
\ref{cor:pruning}.   Recall  that  $\hat \Pi_\text{Lévy}^{[n]}$  is  the
restriction     of      $\Pi_\text{Lévy}$     defined     in     Section
\ref{sec:partition-L} to the $n$ first integers.
\begin{cor}
   \label{cor:GW-CRT}
The process $\hat \Pi_\text{Lévy}^{[n]}$  is under
$\N^{(1)}$ distributed as $\hat
\Pi^{[n]}_\text{GW}$ under $ \bP_n$. 
\end{cor}

Using Lemma  \ref{lem:discrete_pruning}, we also have the
following corollary which shows that the first coalescent event in
$\hat\Pi^{[n]}_\text{L\'evy}$ is not exponentially distributed.

\begin{cor}
\label{cor:first_event}
Let $\tau_1^{(n)}$ be the first coalescent event in
$\hat\Pi^{[n]}_\text{L\'evy}$. Then we have for $\theta\geq 0$:
$$\N^{(1)}[\tau_1^{(n)}\ge \theta]=\left(\frac{1}{1+\theta}\right)^{n-1}.$$
\end{cor}

\begin{proof}
We keep the notations of the proof of Lemma \ref{lem:discrete_pruning}. We have:
\begin{align*}
\N^{(1)}[\tau_1^{(n)}\ge \theta] & = \N\left[\N\left[\inf_{u\in \cn(\tilde
      T_0)}\zeta_u\ge \theta \bigm| \tilde T_0\right] \Bigm| M=n\right]\\
& = \N\left[ \prod_{u\in\cn(\tilde
  T_0)}\left(\frac{1}{1+\theta}\right)^{k_u(\tilde T_0)-1}\Bigm|
M=n\right]\\
& =\N\left[\left(\frac{1}{1+\theta}\right)^{M-1}\Bigm|
M=n\right]\\
& =\left(\frac{1}{1+\theta}\right)^{n-1},
\end{align*}
using \reff{eq:leaves-nodes} for the third equality.
\end{proof}

\section{Proof of Proposition \ref{cor:momentZ}}
\label{sec:proof-cut}

We recall results from \cite{hm:sslnmc}, Corollary 1. Let $X_n$ be
the number of coalescence events for a $\beta(a,b)$-coalescent. For
$1<a<2$ and $b>0$, we have that:
\[
\frac{2-a}{\Gamma(a)} n^{a-2} X_n
\]
converges in distribution towards
\[
W_{a,b}=\int_0^\infty  dt\; \expp{-(2-a) S_{a,b}(t)},
\]
where $S_{a,b}$ is a subordinator with Laplace exponent $\phi_{a,b}$
 given by:
\[
\phi_{a,b}(\lambda)=\int_0^1 \left(1- (1-x)^\lambda \right)
x^{a-3}(1-x)^{b-1} \; dx.
\]
Notice that this notation  is consistent with \reff{eq:def-phi}. 
Since  $Z_n$ is distributed as $X_n$ 
with $a=1+\alpha$ and $b=1-\alpha$. We deduce that:
\[
n^{\alpha-1} Z_n\xrightarrow[n\rightarrow+\infty ] {(d) }Z,
\]
with $Z$ distributed as $\frac{\Gamma(1+\alpha)}{1-\alpha} W_{1+\alpha,1-\alpha}$.

Using  Lemma \ref{lem:f(l)}, 
we compute the moments of $Z$:
\begin{align*}
   \E\left[W_{1+\alpha, 1-\alpha}^n\right]
&=n! \int_{0\leq t_1\leq \cdots\leq t_n}\E\left[ \expp{- (1-\alpha)\sum_{k=1}^n
  S_{1-\alpha, 1+\alpha}(t_k)} \right]\; dt_1 \cdots dt_n\\
&=n! \int_{0\leq r_1, \cdots, 0 \leq r_n}\prod_{k=1}^n \E\left[ \expp{-
    (1-\alpha) 
  kS_{1-\alpha, 1+\alpha}(r_k)}\right] \; dr_1 \cdots dr_n\\
&=\frac{n!}{\prod_{k=1}^n \phi_{1+\alpha,1-\alpha}(k(1-\alpha))}\\
&=\left(\frac{1-\alpha}{\Gamma(\alpha)}\right)^n \frac{\Gamma(n+1)
  \Gamma(1-\alpha)}{\Gamma((n+1)(1-\alpha))}\cdot 
\end{align*}
We deduce that:
\[
\E\left[Z^n\right]=\alpha^n
\, \frac{\Gamma(n+1) \Gamma(1-\alpha)}{\Gamma((n+1)(1-\alpha))}\cdot
\]

\section{Number of blocks in the last coalescence event}
\label{sec:proof-bn}

We consider  the  number of  blocks $B_n$  involved in  the  last coalescence
event of $\Pi^{[n]}_\text{dis}$. 
In order to stress the dependence in $n$,
we shall denote by $T_n$ the GW  tree $T$
under $\bP_n$. We also write $\xi_u(T_n)$ for $\xi_u$ to stress the
dependence of the marks introduced in Section \ref{sec:d-tree-valued} as
a function of the 
underlying tree $T_n$. 
Notice that the time $\xi_\emptyset(T_n)$ at which the
root of $T_n$ is marked corresponds to the last coalescence event
associated with $T_n$. 
Thanks to   Theorem \ref{th:main}, $B_n$
is distributed as the number of leaves of the pruned tree obtained from $T_n$
just before the last coalescence event, that is:
\begin{equation}
   \label{eq:def-bn}
B_n \overset{(\text{d})}{=} L(\cp_{\xi_\emptyset(T_n)-}(T_n)).
\end{equation}

\subsection{Local limit}
The method used in
\cite{ad:cbcpbt} when $\alpha=1/2$ relies on the Aldous's CRT, which is
the (global) limit of $T_n$ when the length of the branch of $T_n$ are
rescaled by $1/\sqrt{n}$, see \cite{d:ltcpcgwt}. Since Lévy's trees are
more difficult to handle, we choose here to use the local limit of
$T_n$, which is the Kesten's tree $T^*$, according  to  \cite{ck:rncpccgwta} Theorem 3.1 or \cite{ad:llcgwtisc} Proposition 4.6.

Recall that $\nu_g$ is the distribution with generating function $g$ given in
\reff{eq:def-g} and  that $\nu_g$ is  critical as $g'(1)=1$.   We recall
the distribution of  the Kesten's tree $T^*$ associated with the critical
reproduction  law $\nu_g$,  see  \cite{k:sbrwrc}. Let  $\nu_g^*$ be  the
corresponding  size-biased distribution: $\nu_g^*(k)=k\nu_g(k)$  for all
$k\in \N$.  For $h\in \N$,  we consider the truncation operator $r_h$ on
$\T$ defined as:
\[
r_h\bt=\{u\in \bt; |u|\leq h\}.
\]
The distribution of $T^*$ is as follows.
Almost surely, $T^*$ contains a unique infinite path
  i.e. a unique infinite sequence  $(V_k, k\in \N^*)$ of positive
  integers such that, for every $h\in \N$, $V_1\cdots V_h\in T ^*$, with
  the convention that $V_1\cdots V_h=\emptyset$ if $h=0$.
The joint distribution of $(V_k, k\in \N^*)$ and $T^*$ is
  determined recursively as follows: for each $h\in \N$,
  conditionally given $(V_1,\ldots,V_h)$ and $r_hT^*$, we have:
\begin{itemize}
\item The number of children $(k_v(T^*),\ v\in T^*,\ |v|=h)$ are
  independent and  distributed according to $\nu_g$ if $v\ne
  V_1\cdots V_h$ and according to $\nu_g^*$ if $v=V_1\cdots V_h$.
\item  Given  also  the  numbers  of children  $(k_v(T^*),\  v\in  T^*,\
  |v|=h)$, the vertex  $V_{h+1}$ is uniformly distributed on  the set of
  integers $\left\{1,\ldots, \sum_{v\in T^*,\ |v|=h}k_v(T^*)\right\}$.
\end{itemize}
We denote by $\P$ the distribution  of $T^*$. 

Recall   that  the   height  of   a   discrete  tree   $\bt\in  \T$   is
$H_\text{max}(\bt)=\sup\{|u|, u\in  \bt\}$. The local  limit convergence
of critical  GW trees,  see \cite{ad:llcgwtisc},  implies that,  for all
$h\in \N^*$, $\bt\in \T$ with height $h$:
\[
\lim_{n\rightarrow+\infty } \bP_n(r_hT_n=\bt)=\P(r_h T^*=\bt).
\]

Notice that  $\cp_\theta(T^*)$ is  a.s. finite  for any  $\theta>0$.  By
construction  of the  marks,  we  easily get  that  the  local limit  of
$(\cp_\theta(T_n),  \theta\geq   0)$  is  given   by  $(\cp_\theta(T^*),
\theta\geq  0)$.    Since  $k_\emptyset(T_n)$   converges  in   distribution  to
$k_\emptyset(T^*)$   (with  distribution   $\nu^*_g$),  we   deduce  the
convergence  in   distribution  of  the  mark   $\xi_\emptyset(T_n)$  to
$\xi^*_\emptyset$ distributed under $\P$ as:
\[
\P(\xi^*_\emptyset\geq \theta |T^*)=(1+\theta)^{1- k_\emptyset(T^*)}.
\]
We deduce that the local limit in distribution of 
$\cp_{\xi_\emptyset(T_n)-}(T_n)$ is given by
$\cp_{\xi_\emptyset^*-}(T^*)$. 

This and the definition of $T^*$ gives the following Lemma.  For $\bt\in
\T$,  and $u\in  \bt$,  recall  the notation  $\bt_u$  for the  sub-tree
attached at $u$, see \reff{eq:def-tu}.

\begin{lem}
   \label{lem:bar-T}
We have, for all $\bt\in \T$:
\[
\lim_{n\rightarrow+\infty } \bP_n ( \cp_{\xi_\emptyset(T_n)-}(T_n)=\bt)
= \P(\bar T=\bt),
\]
where $\bar T$ is  such that:
\begin{itemize}
   \item $k_\emptyset(\bar T)$ has distribution $\nu_g^*$.
   \item Conditionally on $k_\emptyset(\bar T)$, $\xi$ is a random variable such that
$\P(\xi\ge \theta)=(1+\theta)^{1-k_\emptyset(\bar T)}$ for all $\theta\geq 0$.
   \item    Conditionally on $k_\emptyset(\bar T)$ and $\xi$, $V_1$ is a
     uniform random variable on $\{1, \ldots, k_\emptyset(\bar T)\}$.
   \item Conditionally on $k_\emptyset(\bar T)$,  $\xi$ and $V_1$, 
      $(\bar T_u, u\in \{1, \ldots, k_\emptyset(\bar T)\})$ are 
      independent random trees distributed such that for $u\neq V_1$,
      $T_u$ is  distributed as $\cp_\xi(T)$ with $T$ a GW 
      tree with offspring distribution $\nu_g$, and $T_{V_1}$ is
      distributed as $\cp_\xi(T^*)$,
     with $T^*$ distributed as the Kesten's tree associated with the
     reproduction   law $\nu_g$. 
\end{itemize}
\end{lem} 
Notice that by construction, $\bar T$ is finite.

\subsection{Proof of Proposition \ref{prop:cv-bn}}

   We deduce from \reff{eq:def-bn}, Lemma \ref{lem:bar-T} and the
   fact that $\bar T$ is a.s. finite, that $B_n$ converge in
   distribution to $B=L(\bar 
   T)$. From Lemma \ref{lem:bar-T}, we have that $B$ is
   distributed as
\[
L(\cp_\xi(T^*))+ \sum_{k=1}^{k_\emptyset -1} L(\cp_\xi(T_k)),
\]
where  $k_\emptyset$  has  distribution  $\nu^*_g$,  $\xi$  has  density
$(k_\emptyset   -1)  (1+\theta)^{-k_\emptyset}\ind_{\{\theta\geq  0\}}$,
$T^*$ is independent and distributed  as the Kesten's tree associated with
$\nu_g$, and $(T_k, k\in \N^*)$ are independent and distributed as a
Galton-Watson tree $T$ with offspring distribution $\nu_g$.  We deduce that:
\[
\E\left[r^B\right]=\E\left[N(N-1) \int_0^{+\infty } (1+\theta)^{-N}
  d\theta \,  \E\left[r^{L_\theta}\right]^{N-1}
  \E\left[r^{L^*_\theta}\right]\right], 
\]
where $N$ has distribution $\nu_g$, $L_\theta$ is the number of leaves of
$\cp_\theta(T)$ and $L_\theta^*$ is the number of leaves of
$\cp_\theta(T^*)$. 

Let $h_\theta$ be the generating function of $L_\theta$ and $h_\theta^*$
be the generating function of $L_\theta^*$. We have:
\[
\E\left[r^B\right]=
\int_0^{+\infty } \frac{d\theta}{(1+\theta)^2} \;
g''\left(\frac{h_\theta(r)}{1+\theta} \right) h_\theta(r) h^*_\theta(r).
\]

Recall that $\cp_\theta(T)$ is a GW  tree whose reproduction
law has generating function $g_\theta$ given by
\reff{eq:def-gq}. Similar arguments as in the proof of
\reff{eq:gen-leaf}, yields that:
\begin{equation}
   \label{eq:gh-q}
g_\theta(h_\theta(r))-h_\theta(r)=g_\theta(0)(1-r).
\end{equation}
We deduce from \reff{eq:def-gq} that:
\[
g''_\theta(r)=g''\left(\frac{r}{1+\theta}\right) \inv{1+\theta}\cdot
\]
We deduce from \reff{eq:gh-q} that:
\begin{equation}
   \label{eq:gq'}
(1-g'_\theta(h_\theta(r))) = \frac{g_\theta(0)}{h'_\theta(r)}
\quad\text{and}\quad g''_\theta(h_\theta(r)) = (1-g'_\theta(h_\theta(r)))
\frac{h''_\theta(r)}{(h'_\theta(r))^2}\cdot
\end{equation}
We obtain:
\[
g''\left(\frac{h_\theta(r)}{1+\theta}\right) \inv{1+\theta}= g_\theta(0)
\frac{h''_\theta(r)}{(h'_\theta(r))^3} \cdot
\]

We now compute $h_\theta^*$. According to Remark 3.7 in
\cite{adh:pcrtst}, we have for $\bt\in \T$:
 \[
\bP( \cp_\theta(T^*)=\bt)= L(\bt)\, \frac{1-g'_\theta(1)}{g'_\theta(0)}
\, 
\bP( \cp_\theta(T)=\bt).
\]
We deduce that:
\begin{align*}
   h^*_\theta(r)
= \bE\left[r^{L^*_\theta}\right]
&= \sum_{\bt\in \T} r^{L(\bt)} \bP( \cp_\theta(T^*)=\bt)\\
&= \frac{1-g'_\theta(1)}{g_\theta(0)} \sum_{\bt\in \T} L(\bt)
r^{L(\bt)} \bP( \cp_\theta(T)=\bt)\\
&= r \, \frac{h'_\theta(r)}{h'_\theta(1)} ,
\end{align*}
where we used the first equality in \reff{eq:gq'} with $r=1$ and
$h_\theta(1)=1$. 
We get:
\begin{equation}
   \label{eq:rBq}
\E\left[r^B\right]=
r \int_0^{+\infty } \frac{d\theta}{1+\theta}
\;\frac{g_\theta(0)}{h'_\theta(1)} \; 
\frac{h''_\theta(r)}{(h'_\theta(r))^2} \; h_\theta(r) . 
\end{equation}

We have from \reff{eq:def-gq} that:
\[
g_\theta(0)= \alpha (1+\theta) \left[1-
      \left(\frac{\theta}{1+\theta}\right)^{1/\alpha} 
    \right].
\]

We deduce from \reff{eq:gh-q} that:
\[
h_\theta(r)=(1+\theta)\left[1-\left\{1-r\left[1-
      \left(\frac{\theta}{1+\theta}\right)^{1/\alpha} 
    \right]\right\}^\alpha\right].
\]

Then,   the  change  of  variable  $x=1-(\theta/(1+\theta))^{1/\alpha}$  in
\reff{eq:rBq} gives  that $\varphi_\alpha$, given  in \reff{eq:def-f-a},
is the generating function of $B$.


\bibliographystyle{abbrv}
\bibliography{biblio}

\end{document}